\documentclass[11pt, reqno]{amsart}    

\usepackage{amsmath,amssymb,latexsym}
\usepackage{graphicx,xcolor,mathrsfs}
\usepackage{cite}

\usepackage[font={scriptsize}]{caption}
\usepackage{extarrows}

\newcommand{\F}{\mathcal{F}}

\newcommand{\R}{\mathbb{R}}

\newcommand{\N}{\mathbb{N}}

\newcommand{\I}{\mathrm{i}}
\newcommand{\bigO}{\mathcal{O}}
\newcommand{\diff}{\mathrm{d}}
\newcommand{\Diff}{\mathrm{D}}

\newtheorem{theorem}{Theorem}[section]
\newtheorem{lemma}[theorem]{Lemma}
\newtheorem{proposition}[theorem]{Proposition}

\newtheorem{corollary}[theorem]{Corollary}
\newtheorem*{main-theorem}{Main Theorem}
\newtheorem*{remark*}{Remark}

\numberwithin{equation}{section}

\usepackage[colorlinks=true]{hyperref}
\hypersetup{urlcolor=blue, citecolor=red, linkcolor=blue}

\begin{document}

\allowdisplaybreaks

\title[Enhanced existence time in fKdV]{Enhanced existence time of solutions to evolution equations of Whitham type}

\author{Mats Ehrnstr\"om}
\author{Yuexun Wang}

\address{Department of Mathematical Sciences, NTNU Norwegian University of Science and Technology, 7491 Trondheim, Norway.}
\email{mats.ehrnstrom@ntnu.no}

 \address{School of Mathematics and Statistics, Lanzhou University, 370000 Lanzhou,
People's Republic of China.}
\address{Universit\' e Paris-Saclay, CNRS, Laboratoire de Math\'ematiques d'Orsay, 91405 Orsay, France.}
\email{yuexunwang@lzu.edu.cn}

\thanks{Both authors acknowledge the support by grant nos. 231668 and 250070 from the Research Council of Norway. The second author also acknowledges the support of the ANR project ANuI, and the partial support of grant no. 830018 from China.}


\begin{abstract}
We show that Whitham type equations \(u_t + u u_x -\mathcal{L} u_x = 0\), where \(L\) is a general Fourier multiplier operator of order \(\alpha \in [-1,1]\), \(\alpha\neq 0\), allow for small solutions to be extended beyond their ordinary existence time.  The result is valid for a range of quadratic dispersive equations with inhomogenous symbols in the dispersive regime given by the parameter \(\alpha\).
\end{abstract}
\maketitle


\section{Introduction}
The enhanced existence time of small solutions to weakly dispersive water wave equations has gained a lot of attention. Going back to Shatah's normal form \cite{MR803256} and the subsequent work of Delort and collaborators on the Klein--Gordon equation \cite{MR2056326}, it has obtained renewed momentum both through the work on the Burgers--Hilbert equation \cite{MR2982741}, but more generally through the analysis related to global well-posedness for the water-wave problem \cite{MR3460636, MR2507638, MR3314514, MR2993751}, as well as in dedicated papers on improved lifespan of solutions to the same problem \cite{MR3535894, MR3962880, MR4246389}.

The problem concerns the extension of small data beyond the typical hyperbolic existence time: for quadratic equations this means existence times of orders \(|u_0|^{-2}\) in the appropriate norm, also sometimes called cubic lifespan by the terms arising in the energy estimates. Such estimates are of natural interest where singularity formation is present, and general well-posedness in an entire regular space is not possible. At the centre lies the interaction of frequencies via multiple wave interactions. The use of the normal form transform gives rise to generally multidimensional symbols that feature singularities where frequency directions interact; by controlling these by means of multilinear estimates in the appropriate norm one obtains the necessary energy estimate.

These questions have been investigated in several settings in the water-wave problem, including  gravitational  \cite{MR3535894} and pure surface tension \cite{MR3667289} waves on infinite depth, arbitrary lifespan on restricted sets of conditions for periodic waves \cite{MR3839269}, extended life span for all parameter values in the periodic gravity-capillary waves \cite{MR4246389}, and more. While the water-wave problem in its full formulation is involved and hardly treatable without exact reformulations using its Hamiltonian, holomorphic or paralinear structure, several model equations display some of the same frequency and qualitative behaviour. The Burgers--Hilbert equation \cite{MR2982741} and more generally the fractional Korteweg--de~Vries (fKdV) equation \(\partial_tu+u\partial_xu-|D|^\alpha\partial_xu=0\), \(\alpha \in \R\), are examples of this \cite{MR3188389}. The Burgers--Hilbert equation is the case \(\alpha = -1\) in the fKdV scale, the inviscid Burgers equation is  \(\alpha = 0\), and the gravity and capillary water wave problems on infinite depth correspond to  fractional values \(\alpha = \pm\frac{1}{2}\) in terms of modelling and linear dispersion \cite{lannes2013water}. More generally the fKdV equation is a good measure of the balance between nonlinearity and dispersion arising in water wave equations.

In \cite{MR3995034} we showed that the main results from \cite{MR2982741} and  \cite{MR3348783} for the Burgers--Hilbert equation extend to set \(\alpha \in [-1,1] \setminus \{0\}\) in the fKdV scale. The operator \(|D|\) in this equation relates to the case of infinite depth \(h\) in the linear dispersion term \(\xi \tanh(h\xi)\). On unit depth, the symbol is
\begin{equation}\label{eq:dispersion}
\left(\frac{(1+\beta\xi^2)\tanh(\xi)}{\xi}\right)^{\frac{1}{2}}
\end{equation}
where \(\beta = 0\) in the gravity case, and  \(\beta > 0\) for capillary-gravity waves. The corresponding inhomogeneous quadratic equations are called the Whitham and capillary-gravity Whitham equations, and are part of a class of very weakly dispersive nonlinear equations \cite{EGW11, MR4072387}. In this investigation, we extend the results from \cite{MR3995034} to that more general class of inhomogeneous model equations, with the goal of providing simple but still general assumptions on the dispersive operator.  We do so in the spirit of  the setup in \cite{EGW11}, considering the family of Whitham-type equations
\begin{align}\label{eq:Whitham type}
	\partial_tu+u\partial_xu-\mathcal{L}\partial_xu=0,
\end{align}
where the operator \(\mathcal{L}\) is a Fourier multiplier with symbol \(p\),
\[
\mathcal{F}(\mathcal{L}f)(\xi)=p(\xi)\hat{f}(\xi).
\]
Here,
\(
\mathcal{F} f(\xi ) = \int f(x) \mathrm{e}^{-\I x \xi}\, \diff x,
\)
so that \(\partial_x \sim \I \xi\). We shall use \(\Diff = -i \partial_x \sim  \xi\) and \(|\Diff| \sim |\xi|\), and solutions \(u\colon [0,T]\times \R \mapsto \R\) will be considered to be regular. What we need are the following assumptions:\\[-6pt]

\begin{itemize}
\item[(A1)]  The symbol \(p\colon \R \to \R\) is  \(C^2\),  even  and strictly monotone on \([0,+\infty)\);\label{assumption:A1}\\[-6pt]

\item[(A2)]  at the far field,
\begin{align}
	|p^{(i)}(\xi)|\eqsim |\xi|^{\alpha-i},  \  |\xi|\gg 1, \qquad i=0,1,2, \label{assumption:A2}
\end{align}
where \(\alpha \in [-1,1] \setminus \{0\}\);\\[-6pt]

\item[(A3)] and one has a local expansion
\begin{align}
	p(\xi)=p(0)+\xi^{2j_*} \tilde p_{2j_*}(\xi),  \qquad  |\xi|\ll 1, \label{assumption:A3}
\end{align}
where \(j_*\in \mathbb{N}\)\footnote{We use \(\N = \{1,2,3, \ldots\}\). } and \(\tilde p_{2j_*}\)  is locally Lipschitz with \(\tilde p_{2j_*}(0) \neq 0\).\\
\end{itemize}

\noindent In difference to \cite{EGW11} we allow for both negative- and positive-order operators, displaying decay, respectively growth, in \eqref{assumption:A2}. The above assumptions describe the results in terms of regularity, the behaviour of the dispersive symbol at infinity, and the order of its local extremum at the origin. This is in line with a research program to describe nonlocal dispersive equations in terms of a quantifiable properties of their nonlocalities and nonlinearities, see for example \cite{MR4097537,EJMR19,MR4002168}.

As mentioned above, a famous instance of a symbol \(p\) satisfying our assumptions is the linear dispersion for uni-directional water waves  \eqref{eq:dispersion}, with \(\beta = 1\) and \(\beta = 0\) corresponding to  \(\alpha=1/2\) and \(\alpha = - 1/2\), respectively, with \(j_*=1\) near the origin. One can build any even, inhomogeneous, one-sided monotone function with a local extremum at the origin and sublinear decay or growth at infinity to fit the assumptions. Here, the order  \(j_* \geq 1\) of the local extremum at the origin is related to regularity of solutions. Our main theorem is the following.

\begin{theorem}\label{main theorem}
Let \(\alpha \in [-1,1]\setminus\{0\}\) and \(N\geq \max\{3,2j_*-1\}\).  Assume that 
\begin{equation*}
\|u_0\|_{H^N(\mathbb{R})}\leq\varepsilon,
\end{equation*}
for some sufficiently small constant \({\varepsilon}>0\) which depends only on \(\alpha\) and \(N\).
Then there exist a positive number \(T\gtrsim \varepsilon^{-2}\) and a unique solution \(u\) in \(C([-T,T];H^{N}(\mathbb{R}))\cap C^1([-T,T];H^{N-2}(\mathbb{R}))\) of  \eqref{eq:Whitham type} with \(u(0,x)=u_0(x)\) such that 	
\begin{equation*}
\|u\|_{L^\infty([-T,T];H^{N}(\mathbb{R}))}\lesssim \varepsilon.
\end{equation*}
\end{theorem}

The proof is built upon a normal form formulation and the modified energy approach from \cite{MR3348783},  and is a generalisation and simplification of the proof devised in \cite{MR3995034}. Note that for \(\alpha \leq \frac{1}{2}\) the equations are quasilinear and cannot be solved with standard contraction techniques \cite{MR1885293}. Although the role of the parameter \(\alpha\) is not visible in our main theorem, an inspection of the proofs shows that values of \(|\alpha|\) close to the origin yield worse estimates than \(|\alpha|\) close to unit size for both positive and negative values of \(\alpha\), even though \(\alpha = -1\) is the current threshold for the method.  As in \cite{MR3995034}, we follow the method of proof from \cite{MR3348783} up to the point of Lemma~\ref{lemma:equivalent norm} in the paper at hand. We formulate in Lemma~\ref{lemma:m} and Corollary~\ref{cor:m} the general properties of the multiplier needed for the energy estimates later in the paper. These results are similar to other ones in the literature, see for example Lemma~2.3 in \cite{feola2020long}, and are
a fundamental part of the paper. The class of equations has no general integrability structure, and singularities at `low' frequencies and surplus growth at `high' frequencies both have to be dealt with. Explained further in \cite{MR3995034}, we emphasise the commutator introduced in \eqref{eq:tilde F + G}, which is handled in Proposition~\ref{prop: F1+G1 full}. Similar commutators appear in other works (see for example Section 4 in \cite{MR3862598}), and are related to cancellations in the resonant set. In our case integration by parts and splitting of the frequency domain reduce the energy estimates required in the last step, while the final commutator is built on a global transformation of variables in Fourier space, serving for a reduction of two orders of differentiation. We have not been able to find whether the expansion in small variables \((\xi - \eta)/\eta\) and \((\sigma - \eta)/\eta\) that we use in high frequencies has been applied elsewhere. We naturally expect that the methods could be used also on other equations with similar forms of multipliers, see for example \cite{DN19_extended}.

Finally, our results are in line with what is known for these types of equations in the same parameter range \cite{MR3188389,MR3317254,MR3906854}. Local well-posedness follows by classical methods \cite{Abdelouhab1989nonlocal}, whereas it is known to exist (i) small and medium-sized travelling waves that are either smooth \cite{MR3485840,MR3841973,MR4061635} or non-smooth \cite{MR4002168,BD18} and existing for all times, and (ii) time-dependent solutions that are not small, but break down in finite time \cite{MR3682673,SW20}.

The outline of the paper is as follows. In Section~\ref{sec:modified energy} we will introduce the pseudoproduct which yields the normal form for the modified energy. In Section~\ref{sec:m} we analyse the multiplier which constitutes the main problem after the transformation, and in Section~\ref{equivalent norm} use this to prove the equivalence of norms. Section~\ref{energy estimates}, finally, is the heart of the paper. It is almost entirely carried out in Fourier space, and contains the quartic energy estimates necessary to close the argument via a continuity argument.

\section{The modified energy}\label{sec:modified energy}
\noindent Standard theory \cite{MR533234} can be used to show that there exists a positive number \(T\gtrsim \varepsilon^{-1}\) and a unique  solution \(u\in C([0,T];H^N(\R))\) of \eqref{eq:Whitham type}. Therefore, to prove Theorem \ref{main theorem}, we need only to prove an a priori \(H^N(\R)\)-bound for the classical solution \(u\in C([0,T];H^N(\R))\). 
For this we define the modified energy by  
\begin{align}\label{eq:modified k-energy}
	E^{(k)}(t)=\|\partial_x^ku\|_{L^2}^2+2\big(\partial_x^ku, \partial_x^kB(u,u)\big)_2,
\end{align}
where the bilinear form \(B\)  defined as a pseudo-product 
\begin{align}\label{eq:bilinear form}
	\mathcal{F}(B(f_1,f_2))(\xi)=\int_{\R}m(\xi-\eta,\eta)\hat{f}_1(\xi-\eta)\hat{f}_2(\eta)\,\diff \eta,
\end{align}
and the multiplier \(m\) is given by 
\begin{align}\label{eq:multiplier}
	m(\xi-\eta,\eta)=\frac{\xi}{2\big[p(\xi-\eta)(\xi-\eta)+ p(\eta)\eta-p(\xi)\xi
		\big]}.
\end{align}
Thus, the meaning of  \(m\) for two general variables \((a,b)\) is 
\begin{align*}\label{eq:multiplier extended}
m(a,b)=\frac{a+b}{2\big[p(a) a+p(b) b-p(a+b) (a+b)\big]}.
\end{align*}
The modified  energy \eqref{eq:modified k-energy} removes all the cubic terms from the equation and itself satisfies a quartic equation.
\begin{lemma}\label{lemma:quartic form} We have
	\begin{equation}\label{eq:quartic form}
	\begin{aligned}
	\frac{1}{2}\frac{\diff}{\diff t}E^{(k)}(t)
	&=\big(\partial_x^k(-u\partial_xu), \partial_x^kB(u,u)\big)_2 + 2\big(\partial_x^ku, \partial_x^kB(-u\partial_x u,u)\big)_2.
	\end{aligned}
	\end{equation}
\end{lemma}

\begin{proof} We first claim that
	\begin{align}\label{eq:bilinear identity}
	-u\partial_xu-\mathcal{L}\partial_xB(u,u)+B(\mathcal{L}\partial_x u,u)+B(u,\mathcal{L}\partial_x u)=0.
	\end{align}
	Indeed, taking Fourier transform on \eqref{eq:bilinear identity} yields  
	\begin{equation*}
	\begin{aligned}
	&\mathcal{F}\big\{-u\partial_xu-\mathcal{L}\partial_xB(u,u)+B(\mathcal{L}\partial_x u,u)+B(u,\mathcal{L}\partial_x u)\big\}(\xi)\\
	&=\I \int_{\R}\bigg( -\frac{\xi}{2}+m(\xi-\eta,\eta)\left[-p(\xi)\xi+p(\xi-\eta)(\xi-\eta)+ p(\eta)\eta
	\right]\bigg)\hat{u}(\xi-\eta)\hat{u}(\eta)\, \diff \eta\\
	&=0,
	\end{aligned}
	\end{equation*}
	where we have used the definition \eqref{eq:multiplier} of \(m\). From the definition of the modified energy \eqref{eq:modified k-energy} and equation \eqref{eq:Whitham type}, one calculates that
\begin{equation}\label{1}
\begin{aligned}
&\frac{1}{2}\frac{\diff}{\diff t}E^{(k)}(t)\\
&=\big(\partial_x^k\partial_tu,\partial_x^ku)_2+(\partial_x^k\partial_tu, \partial_x^kB(u,u)\big)_2
+\big(\partial_x^ku, \partial_x^k\partial_tB(u,u)\big)_2\\
&=\big(\partial_x^k(\mathcal{L}\partial_xu-u\partial_xu),\partial_x^ku\big)_2+\big(\partial_x^k(\mathcal{L}\partial_xu
-u\partial_xu), \partial_x^kB(u,u)\big)_2\\
&\quad+2\big(\partial_x^ku, \partial_x^kB(\mathcal{L}\partial_x u-u\partial_ x u,u)\big)_2\\
&= -\big(\partial_x^ku,\partial_x^k(u\partial_xu)\big)_2 -\big(\partial_x^ku, \partial_x^k(\mathcal{L}\partial_xB(u,u))\big)_2\\
&\quad +\big(\partial_x^k(-u\partial_xu), \partial_x^kB(u,u)\big)_2 +2\big(\partial_x^ku, \partial_x^kB(\mathcal{L}\partial_x u-u\partial_ x u,u)\big)_2,
\end{aligned}
\end{equation} 	
Since the multiplier \(m\) is symmetric in \(\xi-\eta\) and \(\eta\), so is the bilinear form \(B(f_1, f_2)\) on \(f_1\) and \(f_2\). Thus 
\(B(\mathcal{L}\partial_x u,u)=B(u,\mathcal{L}\partial_x u)\). Finally inserting \eqref{eq:bilinear identity} into \eqref{1} gives  the quartic equation for the evolution of the modified energy.  
	
\end{proof}

Our main two tasks are, on the one hand, to show that the modified energy \(E^{(k)}(t)\) is almost equivalent to the Sobolev energy provided the solution is small in \(H^N(\R)\), on the other hand, to obtain a quartic-type a priori estimate on \(E^{(k)}(t)\). More precisely, we will prove the following two lemmas:    
\begin{lemma}\label{lemma:equivalent norm} 
Let \(\alpha \in [-1,1]\setminus\{0\}\). For any \(N\geq \max\{2,2j_*-1\}\), one has
\begin{align*}
  \sum_{k=2j_*-1}^NE^{(k)}(t)+\|u\|_{L^2}^2=\big(1+\bigO(\|u\|_{H^2})\big)\|u\|_{H^N}^2.
\end{align*}
\end{lemma}
\begin{lemma}\label{lemma:energy estimates} Let \(\alpha \in [-1,1]\setminus\{0\}\). Then
	\begin{align*}
	\frac{\diff}{\diff t}E^{(k)}(t)\lesssim \left( \|u\|_{H^2} \|u\|_{H^3} +  \|u\|_{H^k}^2 \right) \|u\|_{H^k}^2, \quad    k\geq 2j_*-1.
	\end{align*}
\end{lemma}

We will first  show how to prove Theorem \ref{main theorem} by using Lemma \ref{lemma:equivalent norm} and Lemma \ref{lemma:energy estimates} and postpone their proofs to Section \ref{equivalent norm} and Section \ref{energy estimates} respectively. 

\begin{proof}[Proof of Theorem~\ref{main theorem}] In view of Lemma \ref{lemma:equivalent norm}, summing over \(k\) from \(2j_*-1\) to \(N\), one has
	\begin{align*}
	\sum_{k=2j_*-1}^NE^{(k)}(t)\lesssim \sum_{k=2j_*-1}^NE^{(k)}(0)+\int_0^t\|u(s,\cdot)\|_{H^N}^4\,\diff s,
	\end{align*}
	which  in turn yields 
	\begin{equation}\label{2}
	\sum_{k=2j_*-1}^NE^{(k)}(t)+\|u\|_{L^2}^2\lesssim \sum_{k=2j_*-1}^NE^{(k)}(0)+\|u_0\|_{L^2}^2
	+\int_0^t\|u(s,\cdot)\|_{H^N}^4\,\diff  s.
	\end{equation}
	Here,  we have used the \(L^2\)-conservation of solutions to \eqref{eq:Whitham type}.	
	According to Lemma~\ref{lemma:energy estimates}, we on the other hand have
	\begin{align}\label{3}
	\sum_{k=2j_*-1}^NE^{(k)}(t)+\|u\|_{L^2}^2
	\eqsim \frac{1}{2}\|u\|_{H^N}^2
	\end{align}
	for all \(t \geq 0\) and all \(\|u\|_{H^N} < \varepsilon\) sufficiently small. We conclude from \eqref{2} and~\eqref{3} that 
	\begin{align*}
	\|u\|_{H^N}^2\lesssim \|u_0\|_{H^N}^2+\int_0^t\|u(s,\cdot)\|_{H^N}^4\,\diff s,
	\end{align*}
	which finishes the proof by applying Gr\"onwall's inequality or an analogous continuity argument. 
\end{proof}

\section{The bound of the multiplier \(m\)}\label{sec:m}
\noindent In this section we study the singularities and growth of the multiplier \(m\) at `low' and `high' frequencies, respectively.

\begin{lemma}\label{lemma:m}
Let 
\(
\varphi(a,b) = p(a)a + p(b)b - p(a+b)(a+b).
\)
Then
\[
\varphi(a,b) = \varphi(b,a), \quad \varphi(-a,-b) = -\varphi(a,b), \quad \varphi(a,b) = \varphi(-(a+b),b),
\]
and
\[
|\varphi(a, b)| \eqsim \frac{|a b (a+b)|}{r^2} \, \min(r^{2j_*}, 1 + r^{\alpha}),
\]
where \(r=  \sqrt{a^2 + b^2}\).
In particular, \(a = 0\), \(b= 0\) and \(a = -b\) are the only zeros of \(\varphi\).
\end{lemma}
Note that \(|a b (a+b)|r^{-2}\) is proportional to the smallest of \(|a|\), \(|b|\) and \(|a+b|\), which is always present in \(\varphi\). Apart from this factor, it is \(r^{2j_*}\) that determines \(\varphi\) locally, and \(1+ r^{\alpha}\) in the far field.

\begin{proof}[Proof of Lemma \ref{lemma:m}]
As \(p\) is even, it is immediate to verify that \(a = 0\), \(b = 0\) and \(a = -b\) are zeros of \(\varphi\). Similarly, the symmetries \((a,b) \leftrightarrow (b,a)\) and \((-a,b) \leftrightarrow (a-b,b)\), as well as the anti-symmetry \((-a,-b) \leftrightarrow (a,b)\), follow directly from the definition. 

 \begin{figure}
\includegraphics[width=0.5\textwidth]{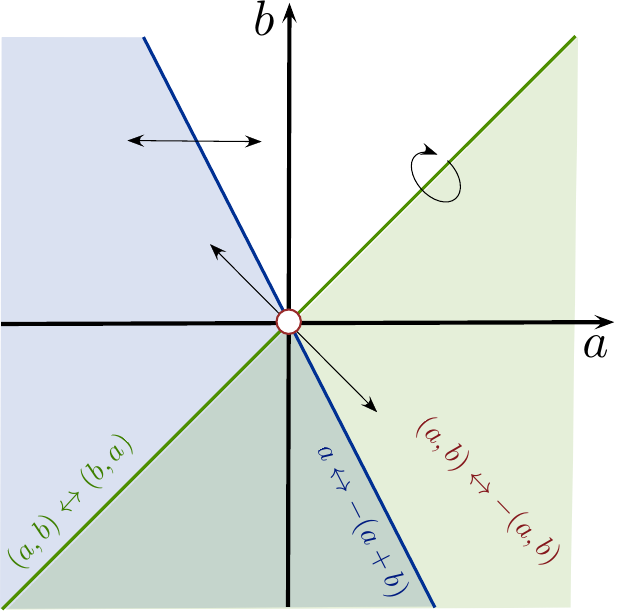}
\caption{The symmetries of the function \(\varphi\).}
\label{fig:phi_diagram}
\end{figure}

To see that there are no other zeros, note first that, by the three above symmetries, it is enough to investigate the region
\[
a = \lambda b, \qquad -\frac{1}{2} < \lambda < 1, \quad \lambda \neq 0,
\] 
in the upper half-plane where \(b > 0\), see Figure~\ref{fig:phi_diagram}. There, we have
\[
\varphi(a,b) = b \underbrace{\left[p(\lambda b) \lambda  + p(b)  - p((1+\lambda)b) (1+\lambda) \right]}_{(\ast)}
\]
Recall that \(p\) is even and strictly monotone on the half-line. In case it is increasing, we have
\[
(\ast) \lessgtr (1+\lambda) \left[ p(b) - p((1+\lambda)b) \right] \lessgtr 0,
\]
where the inequalities vary with the sign of \(\lambda \in (-\frac{1}{2},0) \cup (0,1)\). Again, if \(p\) is instead decreasing, the signs reverse. Hence, \(p\) has no additional zeros, and if we take a \(\frac{\diff}{\diff a}\)-derivative for any fixed \(b > 0\) at the zero \(a = 0\), we get
\[
\frac{\diff}{\diff a} \varphi(a,b)|_{a = 0} =  - p'(b) b - [p(b) - p(0)]  \neq 0,
\] 
because of the strict monotonicity of \(p\). Hence, by the three symmetries, all zeros of \(\varphi\) outside the origin are simple, and \(\varphi\) may be expressed as
\begin{equation}\label{eq:basic zeros}
\varphi(a,b) = a b (a+b) \tilde \varphi(a,b),
\end{equation}
where \(\tilde \varphi\) is continuous and non-vanishing  for \((a,b) \neq (0,0)\). 

We now consider three cases. Let \(r= \sqrt{a^2 + b^2}\), and consider separately the cases 
\begin{itemize}
\item[(i)] \(0 < r \ll 1\),
\item[(ii)] \(|a| \lesssim 1\) with \(r \gg 1\),
\item[(iii)] \(|r| \gg 1\) with \(|a|, |b|, |a+b| \gg 1\).
\end{itemize}

Note that the intermediate cases are already covered by \eqref{eq:basic zeros}, and that (ii) covers the case when either \(b\) or \(a+b\) is substituted for  \(a\) in view of the symmetries of \(\varphi\). Exactly one, or all, of \(a\), \(b\) and \(a+b\) can be small at the same time.

In the case of (i) all terms are small. We assume by symmetry that \(a\) is the smallest in the triad \(a,b, a+b\), and perform the calculation for the variables \(a\) and \(r\), noting that \(r\) is the biggest of them. From the local Lipschitz continuity of the Taylor coefficient \(p_{2j_*}\) we have an expansion
\begin{align*}
\varphi(a,b) &= p(a) a + p(b) b - p(a+b)(a+b)\\ 
&= a[ p(0) + a^{2j_*} p_{2j_*}(a)] + b[ p(0) + b^{2j_*} p_{2j_*}(b)]\\
&\quad - (a+b) [p(0) + (a+b)^{2j_*} p_{2j_*}(a+b)]\\ 
&= a^{2j_*+1} \left(p_{2j_*}(0) + o_r(1)\right) + b^{2j_*+1} p_{2j_*}(b)\\ 
&\quad -  \sum_{j=0}^{2j_* + 1} \binom{2j_*+1}{j} a^j b^{2j_*+1 -j} (p_{2j_*}(b) + a \, o_r(1))\\ 
&= p_{2j_*}(0)  \bigg( -  \sum_{j=1}^{2j_*} \binom{2j_*+1}{j} a^{j-1} b^{2j_*+1 -j}  + o(r^{2j_*}) \bigg) a  \\
&\eqsim -  p_{2j_*}(0)  a r^{2j_*}.
\end{align*}
It is possible to explicitly extract also \(b\) and \(a+b\) in the above expression, but it is not necessary, as we know that the orders of of \(a\), \(b\) and \(a+b\) are identical. Thus, we have
\[
|\varphi(a,b)| \eqsim |ab(a+b)| r^{2j_* -2}, \qquad r \ll 1,
\]
for small frequencies.

In the case of (ii), we instead expand \(p(\cdot + a)\) in \(a\), yielding
\begin{equation}\label{eq:expand in a}
\begin{aligned}
\varphi(a,b) &= p(a) a + p(b) b - p(a+b)(a+b)\\ 
&= p(a) a + p(b)b - \left[p(b) +  a \int_0^1 p'(sa + b) \,\diff s \right] (a+b)\\
&=  a \left[ p(a) - p(b) - (a+b) \int_0^1 p'(sa + b) \,\diff s \right].
\end{aligned}
\end{equation}
As \(p\) is monotone and \(a\) is assumed bounded and \(b\) large, the terms \(p(a)-p(b)\) and \(- (a+b) \int_0^1 p'(sa + b) \,\diff s\) are of the same sign. By monotonicity and the far-field assumption \eqref{assumption:A2}, the distance between \(p\) at the origin and at infinity is bounded and nonzero for \(\alpha < 0\), and infinite for \(\alpha >0\). More precisely,
\[
|p(a)-p(b)| \eqsim 1 + r^{\alpha},
\]
while
\[
\bigg|- (a+b) \int_0^1 p'(sa + b) \,\diff s \bigg| \eqsim r^{\alpha}.
\]
Therefore,
\[ 
|\varphi(a,b)|\eqsim |a| (1 + r^{\alpha}), 
\]
when \(|b| \eqsim r \gg 1\) and \(|a| \lesssim 1\). 

If, as in the case of (iii), none of the variables \(a\), \(b\) and \(a+b\) are bounded, the difference \(p(a) - p(b)\) need not be bounded away from zero. One can see from \eqref{eq:expand in a} that the expression \(|a|(1 + r^\alpha)\) becomes
\[
|\varphi(a,b)|\eqsim \min(|a|,|b|, |a+b|) \max(|a|^\alpha,|b|^\alpha,|a+b|^\alpha),
\]
where we note that the max is determined by different variables depending on the sign of \(\alpha\).
By combining (ii) and (iii) we thus get the global behaviour
\[
|\varphi(a,b)| \eqsim \frac{|ab(a+b)|}{r^2}  \left( (1+|a|)^\alpha + (1+|b|)^\alpha \right), \qquad r \gtrsim 1,
\]
for `large' frequencies (either \(a\), \(b\) or \(a+b\) could still be small).
\end{proof}

The following corollary is immediate from Lemma \ref{lemma:m}. It describes the low-frequency singularities and high-frequency derivatives of the symbol \(m\). 
\begin{corollary}\label{cor:m}
The symbol \(m\) satisfies
\[
|m(\xi-\eta,\eta)| \eqsim \frac{1}{|\eta (\xi - \eta)|} \left[\frac{1}{r^{2j_* - 2}} + \frac{r^2}{(1+ |\eta|)^\alpha + (1+ |\xi - \eta|)^\alpha} \right], 
\] 
where \(r^2\) could be replaced by any two-term sum of \(\xi^2\), \(\eta^2\) and \((\xi - \eta)^2\). In particular, 
\begin{align*}
|m(\xi-\eta,\eta)| &\lesssim \underbrace{\frac{1}{|\eta (\xi - \eta) \xi^{2j_* - 2}|}}_{m_1(\xi-\eta,\eta)}  + \underbrace{1 + |\xi|\left(1 + \frac{1}{|\eta|} + \frac{1}{|\xi -\eta|}\right)}_{m_2(\xi-\eta, \eta)}.
\end{align*} 
\end{corollary}

Note that although Corollary~\ref{cor:m} gives a precise equivalence for \(m\), the estimate that follows from it is in fact enough to treat all the cases \(\alpha \in (-1,0) \cup (0,1)\) in a uniform way. The singularities in the symbol still provide some difficulties, and Sections~\ref{equivalent norm} and~\ref{energy estimates} are mainly devoted to treat these in the energy estimates.

\section{The proof of Lemma \ref{lemma:equivalent norm}}\label{equivalent norm}	
\noindent  This section is devoted to the proof of the equivalent energy norm defined in Lemma~\ref{lemma:equivalent norm}. Let \(k\geq 2j_*-1\). We will be done if we can show that 
\begin{equation}\label{eq:cubic estimate}
\big(\partial_x^ku, \partial_x^kB(u,u)\big)_2 \lesssim \|u\|_{H^2}\|u\|_{H^{k}}^2.
\end{equation}
Using the symmetry of \(B\) one may write  
	\begin{align*}
		&\big(\partial_x^ku, \partial_x^kB(u,u)\big)_2\\
		&=2\int_{\mathbb{R}^2}m(\xi-\eta,\eta)\hat{u}(\xi-\eta) (\I \eta)^k \hat{u}(\eta) \overline{(\I \xi)^k \hat{u}(\xi)}\,\diff\eta\, \diff \xi\\
		&\quad+\sum_{j=j_*}^{k-1} c_{k,j} \int_{\mathbb{R}^2}m(\xi-\eta,\eta)(\I(\xi-\eta))^{j}\hat{u}(\xi-\eta) (\I \eta)^{k-j} \hat{u}(\eta) \overline{(\I \xi)^k \hat{u}(\xi)}\,\diff\eta\, \diff \xi\\
		&=\colon 2A_0+\sum_{j=j_*}^{k-1}c_{k,j}A_j,
	\end{align*}
where \(c_{k,j}\) are binomial coefficients. The term \(A_0\) possesses a singularity at low frequencies and an extra derivative at high frequencies, and we handle it first.  We will use integration by parts to eliminate the worst part of \(A_0\) and thus reduce it.  Note that on the Fourier side, integration by parts corresponds to formula \(-\I\xi = -(\I(\xi - \eta) + \I\eta)\), hence
\begin{equation}\label{18}
\begin{aligned}
A_0 &=\underbrace{-\int_{\mathbb{R}^2}m(\xi-\eta,\eta) \I (\xi - \eta) \hat{u}(\xi - \eta) (\I \eta)^k \hat{u}(\eta) \overline{(\I \xi)^{k-1} \hat{u}(\xi)}\,\diff\eta\, \diff \xi}_{A_{0}^1\colon \text{good part}}\\
	&\quad\underbrace{-\int_{\mathbb{R}^2}m(\xi-\eta,\eta) \hat{u}(\xi - \eta) (\I \eta)^{k+1} \hat{u}(\eta) \overline{(\I \xi)^{k-1} \hat{u}(\xi)}\,\diff\eta\, \diff \xi}_{A_{0}^2\colon \text{bad part}},
\end{aligned}
\end{equation}
in which we have utilised the fact  \(\overline{\hat u(\xi)} = \hat u(-\xi)\) since \(u\) is real  and \(m\) is invariant under the map \((\xi, \eta) \mapsto -(\xi, \eta)\).
We next perform the change of variables \(\xi \leftrightarrow \eta\) on \(A_{0}^2\) and note that the solution \(u\) is real, to find that  
\begin{align*}
	A_{0}^2
	=-\int_{\mathbb{R}^2}m(\eta-\xi,\xi)\hat{u}(\xi-\eta) (\I \eta)^{k-1} \hat u(\eta) \overline{(\I \xi)^{k+1} \hat u(\xi)} \,\diff\eta \, \diff \xi.
\end{align*} 
Therefore  
\begin{equation}\label{19}
\begin{aligned}
	A_0+A_{0}^2
	&=\int_{\mathbb{R}^2}\eta^{-1}\big[m(\xi-\eta,\eta)\eta+m(\eta-\xi,\xi)\xi\big]\\
	&\quad\times \hat{u}(\xi-\eta)(\I \eta)^k \hat{u}(\eta) \overline{(\I \xi)^k \hat{u}(\xi)}\,\diff\eta\, \diff \xi =0,
\end{aligned}
\end{equation}
where we have applied the equality  
\begin{align*}
m(\xi-\eta,\eta)\eta+m(\eta-\xi,\xi)\xi=0.
\end{align*}
Inserting \eqref{19} to \eqref{18}, we finally may express 
\begin{align}\label{20}
A_0=\frac{1}{2}A_{0}^1.
\end{align}

Hence, it is sufficient to estimate the terms \(A_0^1\) and \(A_j\), \(j=1,\cdots,k-1\). In view of Corollary \ref{cor:m} the multiplier \(m\) is controlled by \(m_1+m_2\). On the one hand,  the factor \( (\xi - \eta)  \eta^k \xi^{k-1}\) appearing in \(A_0^1\) and \(A_j\)  eliminates the low-frequency singularities of \(m_1(\xi-\eta,\eta)\) because of \(k \geq 2j_*-1\); on the other hand, the multiplier \(m_1\) does not contain any derivatives (high-frequency). We let \(A_0^1(m_1)\) denote the part controlled by \(m_1\) of the integral, and similarly for other integrals to come.
Consequently we may  estimate  
\begin{equation}\label{21}
\begin{aligned}
A_0^1(m_1)&\lesssim\bigr\||D|^{-1} \partial_xu\bigr\|_{L^2} \bigr\| |D|^{-1} \partial_x^ku \bigr\|_{H^1}
\bigr\||D|^{-(2j_*-2)}\partial_x^{k-1}u\bigr\|_{L^2}\\
&\lesssim\|u\|_{L^2}\|u\|_{H^{k}}\|u\|_{H^{k-2j_*+1}}\leq\|u\|_{L^2}\|u\|_{H^{k}}^2,
\end{aligned}
\end{equation}
and
\begin{equation}\label{22}
\begin{aligned}
A_j(m_1)&\lesssim\bigr\||D|^{-1}\partial_x^j u\bigr\|_{L^2}\bigr\||D|^{-1}  \partial_x^{k-j} u\bigr\|_{H^1} \bigr\||D|^{-(2j_*-2)}\partial_x^{k} u\bigr\|_{L^2}\\
&\lesssim\|u\|_{H^{j-1}}\|u\|_{H^{k-j}}\|u\|_{H^{k-2j_*+2}}\leq \|u\|_{H^{k}}^3.
\end{aligned}
\end{equation}

For the parts \(A_0^1(m_2)\) and \(A_j(m_2)\), there is also no singularity at low frequencies in the total symbol, thus we need only to focus on the high-frequency derivatives. The estimate on \(A_0^1(m_2)\) is straightforward: 
\begin{equation}\label{23}
\begin{aligned}
A_0^1(m_2)&\lesssim\bigr\| \partial_x u\bigr\|_{H^1} \bigr\| \partial_x^ku \bigr\|_{L^2}
\bigr\|\partial_x^{k-1}u\bigr\|_{L^2}\\
&\quad+\bigr\||D|^{-1}\partial_xu\bigr\|_{H^1}\bigr\|\partial_x^ku\bigr\|_{L^2} 	\bigr\||D|\partial_x^{k-1}u\bigr\|_{L^2}\\
&\quad+\bigr\| \partial_xu\bigr\|_{L^2}\bigr\||D|^{-1}\partial_x^ku\bigr\|_{H^1} \bigr\||D|\partial_x^{k-1}u\bigr\|_{L^2}\\
&\lesssim\|u\|_{H^2}\|u\|_{H^{k}}^2.
\end{aligned}
\end{equation}
To handle \(A_j(m_2)\), we shall distribute the derivative in the total symbol to its correct position. Before going ahead, we make a general rehearsal for convenience hereafter.    
\begin{lemma}\label{le:rehearsal}   
\begin{align*}
T(\hat{f},\hat{g})(a)\colon=\int_{\R}m_2(a-b,b)\hat{f}(a-b)\hat{g}(b)\, \diff b.
\end{align*}
Then 
\begin{equation*}
\begin{aligned}
\|T(\hat{f},\hat{g})\|_{L^2}&\lesssim \|(1+|D|)f\|_{L^2}\|(1+|D|^{-1})g\|_{H^1}\\
&\quad+\|(1+|D|^{-1})f\|_{H^1}\|(1+|D|)g\|_{L^2}.
\end{aligned}
\end{equation*}
\end{lemma}
\begin{proof}
The desired result follows from 
\begin{equation*}
\begin{aligned}
m_2(a-b,b)
&= 1 + \frac{|a-b|}{|b|} +\frac{|b|}{|a-b|}  \\
&\lesssim (1+|a-b|)(1+\frac{1}{|b|})+(1+|b|)(1+\frac{1}{|a-b|}),
\end{aligned}
\end{equation*}
and
\begin{equation*}
\begin{aligned}
\quad\left\|\int \big|\hat{f}(a-b)\widehat{g}(b)\big|\, \diff b\right\|_{L_a^2}
\leq \int \left(\int \big|\hat{f}(a-b)\big|^2\, \diff a\right)^{1/2} \big|\widehat{g}(b)\big|\, \diff b
\lesssim \|f\|_{L^2}\|g\|_{H^1}.
\end{aligned}
\end{equation*}
\end{proof}

From Lemma \ref{le:rehearsal}, it follows that 
\begin{equation}\label{24}
\begin{aligned}
A_j(m_2)&\lesssim \bigr\|\partial_x^{k}u\bigr\|_{L^2}\bigg(\bigr\|(1+|D|)\partial_x^j u\bigr\|_{L^2} \bigr\|(1+|D|^{-1})  \partial_x^{k-j} u\bigr\|_{H^1}\\
&\quad +\bigr\|(1+|D|^{-1}) \partial_x^j u\bigr\|_{H^1}\bigr\|(1+|D|)\partial_x^{k-j} u\bigr\|_{L^2}\bigg) \\
&\lesssim \|u\|_{H^k}^3.
\end{aligned}
\end{equation}

We conclude  \eqref{eq:cubic estimate} from \eqref{21}-\eqref{24}, which proves Lemma~\ref{lemma:equivalent norm}.

\section{The proof of Lemma \ref{lemma:energy estimates}}\label{energy estimates}
\noindent We will first perform some  basic algebraic manipulations to reduce the evolution \eqref{eq:quartic form} of the modified energy \(E^{(k)}(t)\) 
to certain higher-order terms which posses extra derivatives at high frequencies. Pointwise monomial estimates like the ones in  Section~\ref{equivalent norm} however cannot be used directly to deal with the higher-order terms left (in fact, they are untrue in that formulation if one wants a bound in \(H^k(\R)\)). To tackle the problem, we therefore perform global transformations which include integration by parts and a series of change of variables in spectral space, which finally reduce the most difficult part of the higher-order terms to a commutator that has two orders of gain in the required  Fourier variables via its difference structure. The main difficulty is that, to find the correct commutator, we need to divide the Fourier space \(\R^3\) into different parts and then make several changes of variables in some symmetric domains far away from the low frequencies to avoid the singularities.

\subsection{Reduction of \(\frac{\diff}{\diff t} E^{(k)}(t),k\geq 2j_*-1\)} 
Note that there are two more derivatives on the right-hand side of \eqref{eq:quartic form} compared to the Sobolev norm in Lemma \ref{lemma:equivalent norm}.  We first introduce a new bilinear form to find a cancellation, which reduces the terms to those contain at most one  extra derivative.    
Let \(Q=\partial_x^{-1} B\) and denote by \(n\) the symbol of \(Q\)  so that
\begin{equation}\label{25}
\begin{aligned}
n(\xi-\eta,\eta) =\frac{-\I}{2\big[p(\xi-\eta)(\xi-\eta)+p(\eta)\eta-p(\xi)\xi\big]},
\end{aligned}
\end{equation}
and
\begin{align}\label{estimate:n}
|n(\xi-\eta,\eta)|
\lesssim\frac{1}{|\xi|^{2j_*-1}|\xi-\eta||\eta|}+(1+\frac{1}{|\xi|})(1+\frac{1}{|\xi-\eta|}
+\frac{1}{|\eta|}).
\end{align}

We  decompose \eqref{eq:quartic form} into its highest-order and remainder terms, as
\begin{equation*}
\begin{aligned}
&\big(\partial_x^kB(u,u), \partial_x^k(-u\partial_xu)\big)_2\\
&=\big(\partial_x^{k+1}Q(u,u), \partial_x^k(-u\partial_xu)\big)_2
=2\big(\partial_x^kQ(u,\partial_x u), \partial_x^k(-u\partial_xu)\big)_2\\
&=2\underbrace{\big(Q(u,\partial_x^{k+1}u), \partial_x^k(-u\partial_xu)\big)_2}_{F_0\colon\text{highest-order term}} +
2\sum_{j=1}^k c_{k,j} \underbrace{\big(Q(\partial_x^ju,\partial_x^{k+1-j}u),\partial_x^k(-u\partial_xu)\big)_2}_{F_j\colon\text{remainder terms}},
\end{aligned}
\end{equation*}
and, using integration by parts,
\begin{equation*}
\begin{aligned}
&2\big(\partial_x^kB(u,-u\partial_x u),\partial_x^ku\big)_2
=2\big(\partial_x^kQ(u,u\partial_x u),\partial_x^{k+1}u\big)_2\\
&=2\underbrace{\big(Q(u,\partial_x^k(u\partial_x u)),\partial_x^{k+1}u\big)_2}_{G_0\colon\text{highest-order term}}
+2\sum_{j=1}^k c_{k,j} \underbrace{\big(Q(\partial_x^j u,\partial_x^{k-j}(u\partial_x u)),\partial_x^{k+1}u\big)_2}_{G_j\colon\text{remainder term}}.\\
\end{aligned}
\end{equation*}

Although the two highest-order terms \(F_0\) and \(G_0\) posses the extra derivatives and singularities, the summation of them  completely cancel each other out. Indeed, an easy calculation shows  
\begin{align*}
F_0
=-\int_{\R^2}\overline{n(\eta-\xi,\xi)}\hat{u}(\xi-\eta)\mathcal{F}(\partial_x^k(u\partial_xu))(\eta)
\overline{\mathcal{F}(\partial_x^{k+1}u)(\xi)}\, \diff\eta\, \diff\xi,
\end{align*}
and then 
\begin{equation}\label{27}
\begin{aligned}
F_0+G_0
&=\int_{\R^2}\big[n(\xi-\eta,\eta)-\overline{n(\eta-\xi,\xi)}\big]\hat{u}(\xi-\eta)\\
&\quad \times \mathcal{F}(\partial_x^k(u\partial_xu))(\eta)
\overline{\mathcal{F}(\partial_x^{k+1}u)(\xi)}\, \diff\eta\, \diff\xi=0,
\end{aligned}
\end{equation}
where we have used the equality \(n(\xi-\eta,\eta)=\overline{n(\eta-\xi,\xi)}\) in view of \eqref{25}.

In the reminder terms,  \(F_j\), \(G_j\) for \(j=2,\ldots,k-1\) are easy terms. In fact, using integration by parts, one rewrites 
\[
F_j=\frac{1}{2}\big(B(\partial_x^ju,\partial_x^{k+1-j}u), \partial_x^ku^2\big)_2,\quad G_j=-\frac{1}{2}\big(B(\partial_x^j u,\partial_x^{k+1-j}u^2),\partial_x^{k}u\big)_2,
\]
and then apply Lemma \ref{le:rehearsal} together with the Gagliardo--Nirenberg inequality (for details, see the calculation in \cite{MR3995034}) to obtain
\begin{align}\label{28}
|\sum_{j=2}^{k-1}F_j |+|\sum_{j=2}^{k-1}G_j |
\lesssim \|u\|_{H^k}^4.
\end{align}
By the symmetry of \(Q\) and integration by parts,  one calculates  
	\begin{align*}
	G_k&= \frac{1}{2} (Q( \partial_x u^2,\partial_x^ku),\partial_x^{k+1}u)_2\\
	&= \frac{1}{2} \int_{\R^2} n(\xi-\eta,\eta) (\xi - \eta) \eta^k \xi^{k+1} \widehat{u^2}(\xi - \eta) \hat u(\eta) \overline{\hat u(\xi)} \, \diff\eta\, \diff\xi\\
	&= \frac{1}{2} \int n(\xi-\eta,\eta) (\xi - \eta)^2 \eta^k \xi^{k} \widehat{u^2}(\xi - \eta) \hat u(\eta) \overline{\hat u(\xi)} \, \diff\eta\, \diff\xi\\
	&\quad+ \frac{1}{2} \int_{\R^2} n(\xi-\eta,\eta) (\xi - \eta) \eta^{k+1} \xi^{k} \widehat{u^2}(\xi - \eta) \hat u(\eta) \overline{\hat u(\xi)} \, \diff\eta\, \diff\xi\\
	&= \frac{1}{2} \int_{\R^2} n(\xi-\eta,\eta) (\xi - \eta)^2 \eta^k \xi^{k} \widehat{u^2}(\xi - \eta) \hat u(\eta) \overline{\hat u(\xi)} \, \diff\eta\, \diff\xi\\
	&\quad - \frac{1}{2} (Q( \partial_x u^2,\partial_x^ku),\partial_x^{k+1}u)_2,
	\end{align*}
	where in the last equaility we have taken advantage of the anti-symmetry of \(n\) in \((\xi-\eta,\eta)\), and the fact that \(u\) is real whereas \(n\) is imaginary. Thus
	\begin{align*}\label{estimate:G_k}
	G_k = \frac{1}{2} \int_{\R^2} n(\xi-\eta,\eta) (\xi - \eta)^2 \eta^k \xi^{k} \widehat{u^2}(\xi - \eta) \hat u(\eta) \overline{\hat u(\xi)} \, \diff\eta\, \diff\xi.
	\end{align*}
Therefore via \eqref{estimate:n} we get 
\begin{equation}\label{30}
\begin{aligned}
|G_k|\lesssim\|u\|_{H^2}\|u\|_{H^3}\|u\|_{H^k}^2.
\end{aligned}
\end{equation}

The following lemma summarises \eqref{27}--\eqref{30}.
\begin{lemma}\label{lemma:ddt sim I} We have
	\[
	\frac{1}{2}\frac{\diff}{\diff t} E^{(k)}(t) \eqsim F_1 + G_1 +\bigO(\|u\|_{H^2}\|u\|_{H^3} + \|u\|_{H^k}^2) \|u\|_{H^k}^2,\quad k\geq 2j_*-1.
	\]
\end{lemma}

Note here that the earlier \(F_k=F_1\) by the symmetry of \(B\) and \(Q\), so we need only to focus on \(F_1\) and \(G_1\). However, there is still one extra derivative in both \(F_1\) and \(G_1\),   with which we shall perform  commutator estimates to treat in the next subsection. 

\subsection{Higher-order estimates: \(F_1\) and \(G_1\)}\label{subsec:F10 + G10}
Just as in \cite{MR3995034} we shall divide frequency space into several symmetric regions. The first of these are
\[
\mathcal{A}_1 = \{(\xi,\eta,\sigma)\in\R^3\colon\min\{|\xi|,|\eta|,|\sigma|\}< 1\},
\]
and it complement \(\mathcal{A}_1^c\), in which \(|\xi|,|\eta|,|\sigma| \geq 1\). Also, for convenience and to standardise the calculations, we shall use the measure notation
\[
\diff M(\hat u)= \hat{u}(\xi-\eta)\hat{u}(\eta-\sigma) \hat{u}(\sigma)\overline{\hat{u}(\xi)} \, \diff\xi \, \diff\eta \, \diff\sigma,
\]
for the quartic factor that will appear in many estimates. An estimate that will be used frequently for both \(\mathcal{A}_1\) and some later sets is
\begin{align}\label{pro:I}
|\xi| + |\eta| + |\sigma| \lesssim 1 +  |\xi - \eta| + |\eta - \sigma|.
\end{align}

To employ commutator estimates to handle the terms \(F_1\) and \(G_1\), 
we first split \(F_1\) and \(G_1\) into a low-frequencies part (to eliminate the singularities) and a high-frequencies part (to distribute the derivatives) respectively, and then extract the lower-order parts from the high-frequencies part. 
By integration by parts, we decompose \(F_1\) as 
\begin{equation*}
\begin{aligned}
F_1&=-\int_{\R^3}n(\xi-\eta,\eta)\widehat{\partial_xu}(\xi-\eta)\widehat{\partial_x^{k}u}(\eta) \overline{(\I\xi)^{k}\hat{u}(\xi-\sigma)\widehat{\partial_xu}(\sigma)} \, \diff\xi \, \diff\eta \, \diff\sigma\\
&=\int_{\R^3}m(\xi-\eta,\eta)\widehat{\partial_xu}(\xi-\eta)\widehat{\partial_x^{k}u}(\eta) \overline{(\I\xi)^{k-1}\hat{u}(\xi-\sigma)\widehat{\partial_xu}(\sigma)} \, \diff\xi \, \diff\eta \, \diff\sigma\\
&=\underbrace{\int_{\mathcal{A}_1}m(\xi-\eta,\eta)\widehat{\partial_xu}(\xi-\eta)\widehat{\partial_x^{k}u}(\eta) \overline{(\I\xi)^{k-1}\hat{u}(\xi-\sigma)\widehat{\partial_xu}(\sigma)} \, \diff\xi \, \diff\eta \, \diff\sigma}_{\mathcal{A}_1F_1}\\
&\quad+\underbrace{\int_{\mathcal{A}_1^c}m(\xi-\eta,\eta)\widehat{\partial_xu}(\xi-\eta)\widehat{\partial_x^{k}u}(\eta) \overline{(\I\xi)^{k-1}\hat{u}(\xi-\sigma)\widehat{\partial_xu}(\sigma)} \, \diff\xi \, \diff\eta \, \diff\sigma}_{\mathcal{A}_1^cF_1},
\end{aligned}
\end{equation*}
where we keep the derivative \(\xi^{k-1}\) for eliminating the possible low-frequency singularity at \(\xi\) in \(m(\xi-\eta,\eta)\) of \(\mathcal{A}_1F_1\), 
and then write the last term by Leibniz's law as 
\begin{equation*}
\begin{aligned}
&\mathcal{A}_1^cF_1
=\underbrace{\int_{\mathcal{A}_1^c}m(\xi-\eta,\eta)\widehat{\partial_xu}(\xi-\eta)\widehat{\partial_x^{k}u}(\eta) \overline{\hat{u}(\xi-\sigma)\widehat{\partial_x^{k}u}(\sigma)} \, \diff\xi \, \diff\eta \, \diff\sigma}_{F_{1,0}}\\
&+\sum_{l=1}^{k-1}c_{k-1,l}\underbrace{\int_{\mathcal{A}_1^c}m(\xi-\eta,\eta)\widehat{\partial_xu}(\xi-\eta)\widehat{\partial_x^{k}u}(\eta) \overline{\widehat{\partial_x^{l}u}(\xi-\sigma)\widehat{\partial_x^{k-l}u}(\sigma)} \, \diff\xi \, \diff\eta \, \diff\sigma}_{F_{1,l}\colon \text{lower-order term}},
\end{aligned}
\end{equation*}
where we have used \(\xi^{k-1}=[(\xi-\sigma)+\sigma]^{k-1}=\sum_{l=1}^{k-1}c_{k-1,l}(\xi-\sigma)^l\sigma^{k-l}\).

For \(G_1\), if \(k=1\) (which means \(j_*=1\)), to eliminate the low-frequency singularity, we write 
\begin{equation*}
\begin{aligned}
G_1= \frac{1}{2}(Q(\partial_xu,\partial_x u^2),\partial_x^2u)_2,
\end{aligned}
\end{equation*}
and then obtain
\begin{equation*}
\begin{aligned}
|G_1|\lesssim \|u\|_{H^1}^2\|u\|_{H^2}^2.
\end{aligned}
\end{equation*}
If \(k\geq 2\) (which means \(j_*\geq 2\)), we instead decompose \(G_1\) in the following manner: 
\begin{equation*}
\begin{aligned}
G_1&=\int_{\R^3}n(\xi-\eta,\eta)(\xi-\eta)\sigma\eta^{k-1}\xi^{k+1}\,  \diff M(\hat u)\\
&=-\I\int_{\R^3}m(\xi-\eta,\eta)(\xi-\eta)\sigma\eta^{k-1}\xi^{k}\,  \diff M(\hat u)\\
&=\underbrace{-\I\int_{\mathcal{A}_1}m(\xi-\eta,\eta)(\xi-\eta)\sigma\eta^{k-1}\xi^{k}\,  \diff M(\hat u)}_{\mathcal{A}_1G_1}\\
&\quad\underbrace{-\I\int_{\mathcal{A}_1^c}m(\xi-\eta,\eta)(\xi-\eta)\sigma\eta^{k-1}\xi^{k}\,  \diff M(\hat u)}_{\mathcal{A}_1^cG_1},
\end{aligned}
\end{equation*}
and then express the last term further by Leibniz's law as
\begin{equation*}
\begin{aligned}
&\mathcal{A}_1^cG_1
=\underbrace{-\I\int_{\mathcal{A}_1^c}m(\xi-\eta,\eta)(\xi-\eta)\eta\sigma^{k-1}\xi^{k}\,  \diff M(\hat u)}_{G_{1,0}}\\
&-\I\sum_{l=1}^{k-2}c_{k-1,l}\underbrace{\int_{\mathcal{A}_1^c}m(\xi-\eta,\eta)(\xi-\eta)\eta(\eta-\sigma)^{l}\sigma^{k-l-1}\xi^{k}\,  \diff M(\hat u)}_{G_{1,l}\colon \text{lower-order term}},
\end{aligned}
\end{equation*}
where we have used \(\eta^{k-2}=[(\eta-\sigma)+\sigma]^{k-2}=\sum_{l=1}^{k-2}c_{k-2,l}(\eta-\sigma)^l\sigma^{k-l-2}\).

Since the terms \(F_{1,l}\  (l=1,2,\cdots,k-1)\) and \(G_{1,l}\  (l=1,2,\cdots,k-2)\) have no extra derivative at high frequencies and no singularity at low frequencies, it is straightforward to verify that
\begin{align*}
|\sum_{l=1}^{k-1}F_{1,l} |+|\sum_{l=1}^{k-2}G_{1,l} |
\lesssim \|u\|_{H^2} \|u\|_{H^k}^3.
\end{align*}

We are now left with four terms: \(\mathcal{A}_1 F_1, \mathcal{A}_1 G_1\) and \(F_{1,0},  G_{1,0}\). The former two terms involve low frequencies, the latter two terms only contain high frequencies. We first handle the terms involving low frequencies.  

\begin{lemma}\label{lemma: A_1 estimates}
	We have
	\[
	| \mathcal{A}_1 F_1 | + | \mathcal{A}_1 G_1 | \lesssim \|u\|_{H^1}\|u\|_{H^2}\|u\|_{H^k}^2.
	\]
\end{lemma}

\begin{proof} 
We recall that 
	\begin{equation*}
	\begin{aligned}
	\mathcal{A}_1 F_1=\int_{\mathcal{A}_1}m(\xi-\eta,\eta)\widehat{\partial_xu}(\xi-\eta)\widehat{\partial_x^{k}u}(\eta) \overline{(\I\xi)^{k-1}\hat{u}(\xi-\sigma)\widehat{\partial_xu}(\sigma)} \, \diff\xi \, \diff\eta \, \diff\sigma.
	\end{aligned}
	\end{equation*}
The part \(\mathcal{A}_1 F_1(m_1)\) is straightforward since the factor \((\xi - \eta)  \eta^k \xi^{k-1}\) eliminates the low-frequency singularities. Considering instead \(\mathcal{A}_1 F_1(m_2)\),
the worst case is that all the derivatives \(\xi^{k-1}=\sigma^{k-1}+\sum_{l=1}^{k-1}c_{k-1,l}(\xi-\sigma)^l\sigma^{k-1-l}\) fall on \(\widehat{\partial_xu}(\sigma)\), which yields a full \(\widehat{\partial_x^ku}(\sigma)\). But in \(\mathcal{A}_1\), the triangle inequality \eqref{pro:I} allows us to move the one extra derivative \(\xi\) in \(m_2\)  to \(\xi-\eta\) or \(\xi-\sigma\) to get the desired bound. 

As for
\begin{equation*}
\begin{aligned}
\mathcal{A}_1 G_1=-\I\int_{\mathcal{A}_1}m(\xi-\eta,\eta)(\xi-\eta)\sigma\eta^{k-1}\xi^{k}\,  \diff M(\hat u).
\end{aligned}
\end{equation*}
the difficult part is \(\mathcal{A}_1 G_1(m_2)\), with the worst situation when all the derivatives \(\eta^{k-1}\sigma\) fall on \(\widehat{u}(\sigma)\). This again gives a maximal \(\widehat{\partial_x^ku}(\sigma)\). Noting the definition of \(\diff M(\hat u)\), just as above we invoke \eqref{pro:I} to move the one extra derivative \(\xi\) in \(m_2\)  to \(\xi-\eta\) or \(\eta-\sigma\) to obtain the bound of Lemma~\ref{lemma: A_1 estimates}. 
\end{proof}

We next make the following very useful observation. 
\begin{lemma}\label{lemma: F10 sim G10} 
	One has
	\[
	F_{1,0} =  G_{1,0} + \bigO(\|u\|_{H^2}\|u\|_{H^3}\|u\|_{H^k}^2).
	\]
\end{lemma}

\begin{proof} Note that the set \(\mathcal{A}_1^c\) is invariant under changes of variables among the variables \(\xi\), \(\eta\) and \(\sigma\).
We now apply the changes of variables \(\xi \leftrightarrow \sigma \leftrightarrow \eta\) to get	
	
\begin{equation*}
\begin{aligned}
F_{1,0}
&=\I\int_{\mathcal{A}_1^c}m(\xi-\eta,\eta)(\xi-\eta)\eta^k\sigma^k\hat{u}(\xi-\eta)\hat{u}(\eta) \overline{\hat{u}(\xi-\sigma)\hat{u}(\sigma)} \, \diff\xi \, \diff\eta \, \diff\sigma\\
& =\I\int_{\mathcal{A}_1^c}m(\eta-\sigma,\sigma) (\eta-\sigma) \sigma^k \xi^k  \, \diff M(\hat u),
\end{aligned}
\end{equation*}
and \(\eta \leftrightarrow \xi \leftrightarrow \sigma \leftrightarrow \eta\) to yield	
	\begin{align*}
	G_{1,0} &= -\I \int_{\mathcal{A}_1^c}m(\xi-\eta,\eta) \eta (\xi-\eta) \sigma^{k-1} \xi^k \, \diff M(\hat u)\\
	&= -\I \int_{\mathcal{A}_1^c}m(\sigma-\eta,\eta) \eta (\sigma - \eta) \xi^{k-1} \sigma^k \hat{u}(\sigma-\eta)\hat{u}(\eta-\xi) \hat{u}(\xi)\overline{\hat{u}(\sigma)} \, \diff\xi \, \diff\eta \, \diff\sigma\\
	&= \I \int_{\mathcal{A}_1^c}m(\eta-\sigma,\sigma) \sigma (\eta - \sigma) \xi^{k-1} \sigma^k \hat{u}(\eta-\sigma)\hat{u}(\xi-\eta) \hat{u}(\sigma)\overline{\hat{u}(\xi)} \, \diff\xi \, \diff\eta \, \diff\sigma.
	\end{align*}
In the last equality we have utilized that the fact that \( G_{1,0}\) and \(u\) are real and the equality
\[
m(\sigma - \eta, \eta) \eta (\sigma - \eta) = m(\eta - \sigma, \sigma) \sigma (\eta - \sigma).
\]
We see that
	\[
	F_{1,0} - G_{1,0} = \I \int_{\mathcal{A}_1^c}m(\eta-\sigma,\sigma) (\eta - \sigma) (\xi - \sigma) \xi^{k-1} \sigma^k\, \diff M(\hat u).
	\]
The factor \((\xi - \sigma)\) has a difference structure and may be controlled by \(|\xi - \eta| + |\eta - \sigma|\). 
In \(\mathcal{A}_1^c\), the one extra derivative \(\xi\) in \(m_2\) is then added to  \(\xi^{k-1}\) to yield the \(H^2 \times H^3 \times H^k \times H^k\)-estimate. 
\end{proof}

From now on we will use the expressions  
\begin{equation*}
\begin{aligned}
F_{1,0}
=\I\int_{\mathcal{A}_1^c}m(\eta-\sigma,\sigma) (\eta-\sigma) \sigma^k \xi^k\,  \diff M(\hat u),
\end{aligned}
\end{equation*}
and
\begin{align*}
G_{1,0} = -\I \int_{\mathcal{A}_1^c}m(\xi-\eta,\eta)  (\xi-\eta) \eta\sigma^{k-1} \xi^k\,  \diff M(\hat u).
\end{align*}
Lemma \ref{lemma: F10 sim G10} means that we may write 
\[ 
F_{1,0}=\frac{1}{2} (F_{1,0} + G_{1,0})+\bigO(\|u\|_{H^2} \|u\|_{H^3} \|u\|_{H^k}^2).
\] 
The terms  \(F_{1,0} + G_{1,0}\) and \(F_{1,0}\) are thus formally equivalent, but we will for convenience use \(F_{1,0} + G_{1,0}\) below to obtain a good commutator. We again split the remaining frequency space into two parts, 
\begin{align*}
&\mathcal{A}_2 =\{(\xi,\eta,\sigma)\in {\mathcal A}_1^c  \colon {\textstyle \frac{1}{10}} |z_2| < |z_1-z_2|+|z_2-z_3|,\\
&\quad\quad\quad\quad\quad\quad\quad\quad\quad\quad\quad\quad\quad\quad\quad\quad  \text{for some choice of } z_j  = \xi,\eta,\sigma\},
\end{align*}
and its complement \(\mathcal{A}_2^c\), in which  \(\frac{1}{10} |z_2| \geq |z_1-z_2|+  |z_2-z_3|\) for all choices of \(z_j  = \xi,\eta,\sigma\). As above, we let \(\mathcal{A}_2 F_{1,0}\) denote the restriction of the integral \(F_{1,0}\) to the set \(\mathcal{A}_2\), and similarly for other integrals and sets to come, so that 
\[\F_{1,0} + G_{1,0}=\mathcal{A}_2(F_{1,0} + G_{1,0})+\mathcal{A}_2^c(F_{1,0} + G_{1,0}).\]
Just as in \(\mathcal{A}_1\), the triangle inequality \eqref{pro:I} holds in \(\mathcal{A}_2\), allowing us to move derivatives from \(\xi, \eta\) and \(\sigma\)  to \(\xi - \eta\) and \(\eta - \sigma\), so  the term \(\mathcal{A}_2(F_{1,0} + G_{1,0})\) may be handled separately, with the following resulting estimate.

\begin{lemma}\label{lemma: A_2 estimates}
	We have
	\[
	| {\mathcal A}_2 F_{1,0} | + | {\mathcal A}_2 G_{1,0} | \lesssim \|u\|_{H^2}\|u\|_{H^3}\|u\|_{H^k}^2.
	\]
\end{lemma}

We shall use symmetry to reduce the integrals further. To that aim, let
\[
\mathcal{A}_{2,+}^c = \left\{ (\xi,\eta, \sigma) \in \mathcal{A}_2^c \colon \xi, \eta, \sigma \geq 1\right\},
\]
be the `positive' part of \(\mathcal{A}_{2,+}^c\). Then, in \(\mathcal{A}_{2,+}^c\), we in place of the earlier triangle inequality \eqref{pro:I} obtain the equivalence
\begin{align}\label{pro:II}
\xi\eqsim\sigma\eqsim\eta\eqsim\xi-\eta+\sigma \gtrsim 1,
\end{align}
which again will help us to exchange derivatives. What remains is \(\mathcal{A}_2^c (G_{1,0}+F_{1,0})\). Since both the multiplier \(m\) and the solution \(u\) are real, by the shift of variables \((\xi, \eta, \sigma) \rightarrow -(\xi, \eta, \sigma)\), one observes that 
\[
\mathcal{A}_2^c F_{1,0}=2\mathcal{A}_{2,+}^c F_{1,0}, \quad  \mathcal{A}_2^c G_{1,0}=2\mathcal{A}_{2,+}^c G_{1,0}.
\]
We then decompose 
\(\mathcal{A}_2^c F_{1,0}\) as
\begin{align*}
\mathcal{A}_{2,+}^c F_{1,0}&=\I \int_{\mathcal{A}_{2,+}^c} m(\eta-\sigma,\sigma) (\eta - \sigma) \sigma^k \xi^k \, \diff M(\hat u)\\ 
&= \underbrace{\I \int_{\mathcal{A}_{2,+}^c}\frac{m(\eta-\sigma,\sigma)}{\eta} (\eta - \xi) (\eta - \sigma) \sigma^k \xi^k \, \diff M(\hat u)}_{\mathcal{A}_{2,+}^c F_{1,0}^{(\eta - \xi)}}\\ 
&\quad+\underbrace{ \I \int_{\mathcal{A}_{2,+}^c} \frac{m(\eta-\sigma,\sigma)}{\eta}  (\eta - \sigma) \sigma^k \xi^{k+1} \, \diff M(\hat u)}_{\mathcal{A}_{2,+}^c \tilde F_{1,0}},
\end{align*}
and \(\mathcal{A}_2^c G_{1,0}\) as
\begin{align*}
\mathcal{A}_{2,+}^c G_{1,0}&=-\I \int_{\mathcal{A}_{2,+}^c} m(\xi-\eta,\eta) \eta (\xi-\eta) \sigma^{k-1} \xi^k  \diff M(\hat u)\\
&= \underbrace{-\I \int_{\mathcal{A}_{2,+}^c}  m(\xi-\eta,\eta) (\eta - \sigma) (\xi-\eta) \sigma^{k-1} \xi^k  \diff M(\hat u)}_{\mathcal{A}_{2,+}^c G_{1,0}^{(\eta - \sigma)}}\\
&\quad \underbrace{-\I \int_{\mathcal{A}_{2,+}^c}  m(\xi-\eta,\eta) (\xi-\eta) \sigma^{k} \xi^{k}  \diff M(\hat u)}_{\mathcal{A}_{2,+}^c \tilde G_{1,0}}.
\end{align*}
Note that the symbol \(\frac{m(\eta-\sigma,\sigma)}{\eta}\) is bounded when considering \(\mathcal{A}_{2,+}^c F_{1,0}^{(\eta - \xi)}\). The one extra derivate \(\xi\) in \(\mathcal{A}_{2,+}^c G_{1,0}^{(\eta - \sigma)}\) may also be transmitted to \(|\xi-\eta|+|\eta-\sigma|+\sigma\), so as to give us the following reduction.

\begin{lemma}\label{prop:tilde FG}
\[\mathcal{A}_2^c (F_{1,0}+G_{1,0}) = \mathcal{A}_2^c (\tilde F_{1,0}+\tilde G_{1,0}) + \bigO(\|u\|_{H^2}^2 \|u\|_{H^k}^2),\]
with \(\mathcal{A}_2^c (\tilde F_{1,0}+\tilde G_{1,0})\) given by the integral
\begin{equation}\label{eq:tilde F + G}
\begin{aligned}
& \I \int_{\mathcal{A}_{2,+}^c} \left[ \frac{m(\eta-\sigma,\sigma)}{\eta}  (\eta - \sigma)-\frac{m(\xi-\eta,\eta)}{\xi} (\xi-\eta)\right]  \sigma^{k} \xi^{k+1}\,  \diff M(\hat u).
\end{aligned}
\end{equation}
\end{lemma}

The commutator in \eqref{eq:tilde F + G} contains the cancellation we need, but in order to disclose it, we need to use a Fourier symmetry that is not present in the set \(\mathcal{A}_{2,+}^c\). We therefore as a final division of \(\R^3\) introduce the sets
\[
\mathcal{B}_1=\{(\xi,\eta,\sigma)\in\R^3 \colon \xi,\xi-\eta+\sigma,\sigma \geq 1\},
\]
and
\begin{align*}
\mathcal{B}_2=\big\{(\xi,\eta,\sigma)\in\R^3 \colon |\sigma-\eta|+|\xi-\sigma| &\leq {\textstyle\frac{1}{10}} \xi,\\ 
|\xi-\eta|+|\sigma-\xi| &\leq {\textstyle\frac{1}{10}} \sigma,\\ 
|\eta-\sigma|+|\xi-\eta| &\leq {\textstyle\frac{1}{10}} (\xi-\eta + \sigma) \big\},
\end{align*}
and their intersection \(\mathcal{B} = \mathcal{B}_1 \cap \mathcal{B}_2\). While the triangle inequality \eqref{pro:I} does not in general hold in \(\mathcal{A}_{2,+}^c\), the set \(\mathcal{B}\) has the advantage that in both \(\mathcal{A}_{2,+}^c \setminus \mathcal{B}\) and \(\mathcal{B} \setminus \mathcal{A}_{2,+}^c\), it holds, enabling us to exchange derivatives. Performing the change of variables \(\eta\mapsto \xi-\eta+\sigma\) on the first term in \eqref{eq:tilde F + G}, we may write 
\begin{equation}\label{eq:A cap B}
\begin{aligned}
&\mathcal{A}_2^c (\tilde F_{1,0}+\tilde G_{1,0})\\
&=\underbrace{\I  \int_{\mathcal{A}_{2,+}^c \cap \mathcal{B}} \left[\frac{m(\xi-\eta,\sigma)}{\xi-\eta+\sigma}-\frac{m(\xi-\eta,\eta)}{\xi}\right]  (\xi - \eta) \sigma^k \xi^{k+1} \, \diff M(\hat u)}_{[\mathcal{A}_{2,+}^c \cap \mathcal{B}](\tilde F_{1,0}+\tilde G_{1,0})\colon \text{the main commutator}}\\
&\quad  \underbrace{-\I  \int_{\mathcal{A}_{2,+}^c \setminus \mathcal{B}} m(\xi-\eta,\eta)  (\xi - \eta) \sigma^k \xi^{k} \, \diff M(\hat u)}_{[\mathcal{A}_{2,+}^c \setminus \mathcal{B}](\tilde F_{1,0}+\tilde G_{1,0})}\\
&\quad +\underbrace{\I  \int_{\mathcal{B} \setminus \mathcal{A}_{2,+}^c} \frac{m(\xi-\eta,\sigma)}{\xi-\eta+\sigma}  (\xi - \eta) \sigma^k \xi^{k+1} \diff M(\hat u)}_{[\mathcal{B} \setminus \mathcal{A}_{2,+}^c](\tilde F_{1,0}+\tilde G_{1,0})}.
\end{aligned} 
\end{equation}
By using \eqref{pro:I} exactly in the same way as before, we then arrive at our final expression. The reason for keeping \(\mathcal{A}_{2,+}^c \cap \mathcal{B}\) is both that it contains the most difficult term in the analysis, but also that \(\mathcal{A}_{2,+}^c \cap \mathcal{B}\) contains a symmetry in which \emph{two} orders the terms in the commutator cancel. This will all be made clear in the proof of Prop.~\ref{prop: F1+G1 full}. For now, we have:

\begin{proposition}\label{cor:thin sets} 
 \begin{align}
 \mathcal{A}_{2,+}^c (\tilde F_{1,0}+\tilde G_{1,0})= [ \mathcal{A}_{2,+}^c \cap \mathcal{B}] (\tilde F_{1,0}+\tilde G_{1,0})+\bigO(\|u\|_{H^{2}} \|u\|_{H^{3}} \|u\|_{H^k}^2).
 \end{align}
\end{proposition}

\noindent \textbf{The main commutator.} We are now at the position to estimate the last term \([\mathcal{A}_{2,+}^c \cap \mathcal{B}]( \tilde F_{1,0}+\tilde G_{1,0})\) to complete the proof of Lemma \ref{lemma:energy estimates}. As mentioned above, this is a good commutator that will improve our estimates by two orders of cancellations, allowing us to move derivatives in an advantageous way, and finally kill the additional derivative.

\begin{proposition}\label{prop: F1+G1 full}
\[
|[\mathcal{A}_{2,+}^c \cap \mathcal{B}](\tilde F_{1,0}+\tilde G_{1,0})| \lesssim \|u\|_{H^1}\|u\|_{H^2} \|u\|_{H^k}^2.
\]
\end{proposition}
\begin{proof} 
Let	
\begin{align*}
N(\xi,\eta,\sigma) =\frac{m(\xi-\eta,\eta)}{\xi}
-\frac{m(\xi-\eta,\sigma)}{\xi-\eta+\sigma}.
\end{align*} 
In view of \eqref{eq:multiplier} one then calculates that
\begin{align*}
N(\xi,\eta,\sigma)&=\frac{1}{2\big[p(\xi-\eta)(\xi-\eta)+p(\eta)\eta-p(\xi)\xi\big]}\\
&\quad-\frac{1}{2\big[p(\xi-\eta)(\xi-\eta)+p(\sigma)\sigma - p(\xi-\eta+\sigma)(\xi-\eta+\sigma) \big]}\\
&=\frac{p(\sigma)\sigma - p(\xi-\eta+\sigma)(\xi-\eta+\sigma) - \big[p(\eta)\eta-p(\xi)\xi\big]}{2\big[p(\xi-\eta)(\xi-\eta)+p(\eta)\eta-p(\xi)\xi\big]}\\
&\quad \times \frac{1}{\big[p(\xi-\eta)(\xi-\eta)+p(\sigma)\sigma - p(\xi-\eta+\sigma)(\xi-\eta+\sigma) \big]}.
\end{align*}
If we let
\[
U(\xi,\eta,\sigma)=\big[p(\sigma)\sigma - p(\xi-\eta+\sigma)(\xi-\eta+\sigma)\big] - \big[p(\eta)\eta-p(\xi)\xi\big]
\]
be the numerator above, it follows from the definition of \(m\) that
\begin{align}\label{39}
|N(\xi,\eta,\sigma)|\lesssim \frac{|m(\xi-\eta,\eta) m(\xi-\eta,\sigma)|}{\xi (\xi-\eta+\sigma)}|U(\xi,\eta,\sigma)|.
\end{align}
We first handle \(U(\xi,\eta,\sigma)\). As we are in \(\mathcal{A}_{2,+}^c\), we may write
\begin{equation*}
\xi=(1+\mu)\eta \quad\text{ and }\quad \sigma=(1+\nu)\eta,
\end{equation*}
where \(|\mu|,|\nu|\leq {\textstyle \frac{1}{10}}\) are uniformly small. A Taylor expansion of the \(p\)-terms in \(U\) then reads
\begin{align*}
U(\xi,\eta,\sigma) &= \eta \Big[ p((1+\nu)\eta ) (1+\nu)  - p((1+\mu + \nu)\eta) (1+\mu + \nu) \\ 
&\quad - \left[ p(\eta) -  p((1+\mu)\eta) (1+\mu)  \right] \Big]\\
&= -\eta  \int_{0}^1 \int_0^1  \frac{\diff}{\diff s} \frac{\diff}{\diff r} \Big[ p((1+s \nu + r \mu) \eta) (1 + s\nu + r\mu) \Big]  \diff r \, \diff s\\
&= -\eta  \int_{0}^1 \int_0^1  \Big[ p''((1+s \nu + r \mu) \eta) \mu \nu \eta^2 (1 + s\nu + r\mu)\\ 
&\qquad\qquad\qquad  + 2 p'((1+s \nu + r \mu) \eta) \mu \nu \eta  \Big]  \diff r \, \diff s,
\end{align*}
which according to the asymptotic assumptions on \(p\) may be estimated by
\begin{equation}\label{40}
\begin{aligned}
|U(\xi,\eta,\sigma)| &\lesssim \eta^3 |\mu \nu| \eta^{\alpha -2} + \eta^2 |\mu \nu| \eta^{\alpha-1}\\ 
&= |\xi - \eta| |\sigma- \eta| \eta^{\alpha -1},
\end{aligned}
\end{equation}
where we recall that \(\eta \gtrsim 1\) in \(\mathcal{A}_{2,+}^c\). In effect, two derivatives have been moved from \(\xi \eqsim \sigma\) to \(\xi - \eta\) and \(\eta - \sigma\).

We next treat the quotient with \(m\) in the right-hand side of \eqref{39}. Using the equivalence and positive size of \(\xi\), \(\sigma\), \(\eta\) and \(\xi-\eta+\sigma\), the estimate for \(m\) in Lemma~\ref{lemma:equivalent norm} directly yields
\begin{align*}
&\frac{|m(\xi-\eta,\eta) m(\xi-\eta,\sigma)|}{\xi (\xi-\eta+\sigma)} \eqsim \left| \frac{m(\xi-\eta,\eta)}{\xi } \right|^2\\
&\eqsim \left( \frac{1}{\xi^2 | \xi - \eta|} \left[ \frac{\xi^2}{(1+ |\xi|)^\alpha + (1+ |\xi - \eta|)^\alpha} \right] \right)^2\\
&\eqsim \frac{1}{ | \xi - \eta|^2} \left[ \frac{1}{1 + |\xi|^{2\alpha} +  |\xi - \eta|^{2\alpha}} \right] .
\end{align*}
By multiplying with \eqref{40}, and using \(\xi \eqsim \eta\), we get the combined estimate 
\[
|N(\xi,\eta, \sigma)| \lesssim \frac{|\sigma - \eta|}{|\xi - \eta|} \min\left\{ \xi^{-1\pm \alpha }\right\} \leq  \frac{1}{\xi} \frac{|\sigma - \eta|}{|\xi - \eta|},
\]
which is to be multiplied with \( (\xi - \eta) \sigma^k \xi^{k+1} \) in \eqref{eq:A cap B}. The combined upper bound \(|\sigma - \eta| \sigma^k \xi^k \) on the multiplier in \([\mathcal{A}_{2,+}^c \cap \mathcal{B}](\tilde F_{1,0}+\tilde G_{1,0})\) then gives an \(H^1 \times H^2 \times (H^k)^2\)-estimate, and completes the proof of Proposition \ref{prop: F1+G1 full}. 
\end{proof}

\section*{Acknowledgment}
\noindent The authors are grateful to the referee for the comments and suggestions that helped improve the exposition of the paper.

\end{document}